\documentclass[
final
, nomarks
]{dmtcs-episciences}

\usepackage{float,amsmath,amssymb,amsthm,mathrsfs,graphicx,subfigure}
\usepackage{caption,enumerate}
\usepackage[round,numbers,sort&compress]{natbib}

\usepackage[utf8]{inputenc}
\usepackage{subfigure}

\newtheorem{lemma}{Lemma}[section]

\newtheorem{theorem}{Theorem}[section]

\author{Shanshan Zhang
  \and Xiumei Wang\thanks{Corresponding author. Email address: wangxiumei@zzu.edu.cn.
  Supported by the National Natural Science Foundation of China (Nos. 11801526, 11971445, and 11571323).}
  \and Jinjiang Yuan\thanks{
  Supported by the National Natural Science Foundation of China (No. 11671368).}}
\title[Formatting an article for DMTCS]{Even cycles and perfect matchings\\
in claw-free plane graphs}

\affiliation{
  School of Mathematics and Statistics, Zhengzhou University, Zhengzhou, Henan 450001,  China}
\received{2020-01-31}
\revised{2020-06-02}
\accepted{2020-08-30}
\begin{document}


\publicationdetails{22}{2020}{4}{6}{6062}
\maketitle
\begin{abstract}
\ Lov{\'a}sz showed that
a matching covered graph $G$ has an ear decomposition starting with an arbitrary edge of $G$.
Let $G$ be a graph which has a perfect matching. We call $G$ cycle-nice if for each even cycle $C$ of $G$, $G-V(C)$ has a perfect matching.
If $G$ is a cycle-nice matching covered graph, then $G$ has ear decompositions starting with an arbitrary even cycle of $G$.
In this paper, we characterize cycle-nice claw-free plane graphs.
We show that the only cycle-nice simple 3-connected claw-free plane graphs are
$K_4$, $W_5$ and $\overline C_6$.
Furthermore, every cycle-nice 2-connected claw-free plane graph can be obtained from a graph in the family ${\cal F}$ by a sequence of three types of operations,
where ${\cal F}$ consists of even cycles, a diamond, $K_4$,  and $\overline C_6$.
\end{abstract}

\section{Introduction}
\label{sec:in}
In this paper, all graphs are connected and loopless, but perhaps have multiple edges. We follow the notation and terminology in \cite{Bondy2008} except otherwise stated. We use $V(G)$ and $E(G)$ to denote the vertex set and edge set of $G$, respectively. For an edge $e$ of $G$, if there is another edge whose ends are the same as $e$, then $e$ is called a \emph{multiple edge}; otherwise, $e$ is called a \emph{single edge}.
The \emph{underlying simple graph} of $G$ is the simple spanning subgraph of $G$ obtained from $G$ by first deleting all the edges and then connecting each pair of adjacent vertices by a single edge.
A \emph{perfect matching} of $G$ is a set of
independent edges covering all vertices of $G$. For a connected
plane bipartite graph $G$ with the minimum vertex degree at least
2, a face $f$ of $G$ is said to be a \emph{forcing face} if
$G-V(f)$  has exactly one perfect matching. The concept of forcing
face was first introduced in Che and Chen \cite{Che2008}, which is
a natural generalization of the concept of forcing hexagon of a
hexagonal system introduced in Che and Chen \cite{Che2006}. For
research on forcing faces, see \cite{Che2008, Che2013}. In
particular, Che and Chen \cite{Che2008} presented a
characterization of plane elementary bipartite graphs whose finite
faces are all forcing (by using ear decompositions). Here, a graph
is \emph{elementary} if the union of all its perfect matchings
forms a connected subgraph. A graph $G$ is called
\emph{cycle-forced} if for each even cycle $C$ of $G$,
$G-V(C)$ has exactly one perfect matching. All the cycle-forced
Hamiltonian bipartite graphs and bipartite graphs have been
characterized, see \cite{Kong, Wang2018}.

A subgraph $H$ of $G$ is called \emph{nice} if
$G-V(H)$ has a perfect matching (see \cite{Lovasz1986}). In
particular, when $H$ is a cycle, we call $H$ a \emph{nice cycle}.
A nice subgraph is also called a conformal subgraph in
\cite{Carvalho+2005}, a well-fitted subgraph in
\cite{McCuaig2004}, and a central subgraph in \cite{Robertson1999}.
A graph $G$ is called \emph{cycle-nice} if each even cycle
of $G$ is a nice cycle. Clearly, a cycle-forced graph is also
cycle-nice. Given a proper subgraph $H$ of $G$, an \emph{ear} of $G$ with respect to $H$ is an odd path
of $G$ having both ends, but no interior vertices, in $H$. A graph
$G$ has an \emph{ear decomposition} if $G$ can be represented as
$G'+P_1+\cdots +P_r$, where $G'$ is a subgraph of $G$, $P_1$ is
an ear of $G$ with respect to $G'$, and $P_i$ is an ear of $G$
with respect to $G'+P_1+\cdots +P_{i-1}$ for $2\leq i\leq r$. Ear
decomposition is a powerful tool in the study of the structure of
matchings and the enumeration of matchings \cite{Liu2015,
Lovasz1986}. The idea of ear decomposition occurred first in
Hetyei \cite{Hetyei1964}, and was further developed by Lov{\'a}sz,
Carvalho, etc. \cite{ Carvalho1999, Carvalho2002, Carvalho2005,
Lovasz1983}.
A graph $G$ is \emph{matching covered} if each
edge of $G$ induces a nice subgraph of $G$.

Let $G$ be a matching
covered graph. Lov{\'a}sz \cite{Lovasz1983} showed that, for a
subgraph $G'$ of $G$, $G$ has an ear decomposition starting with
$G'$ if and only if $G'$ is a nice subgraph of $G$. This implies
that, for each edge $e$ of $G$, there is an ear decomposition
starting with $e$, that is, $G=G'+P_1+\cdots +P_r$, where $G'$ is
induced by the edge $e$. In this case $G'+P_1$ is an even cycle.
If $G$ is a cycle-nice matching covered graph, then $G$ has ear
decompositions starting with an arbitrary even cycle.

A graph $G$ is \emph{claw-free} if the underlying simple graph of $G$ contains
no induced subgraph isomorphic to the complete bipartite graph $K_{1,3}$.
Sumner \cite{Sumner1974,Sumner1974',Sumner1976} and Las Vergnas \cite{Vergnas1975}
studied perfect matchings in claw-free graphs, and independently
showed that every connected claw-free graph with an even number of vertices has a perfect matching.
A characterization of  2-connected claw-free cubic graphs  which  have ear decompositions starting with an arbitrary induced even cycle is presented in \cite{PengW2019}.

In this paper we present a characterization of cycle-nice
2-connected claw-free plane graphs. The paper is organized as
follows. In Section 2, we present some basic results. In Section
3, we prove that the only cycle-nice simple 3-connected claw-free
plane  graphs are $K_4$, $W_5$ and $\overline C_6$. In Section 4,
we show that every cycle-nice 2-connected claw-free plane graph
can be obtained from a graph in the family ${\cal F}$ by a
sequence of three types of operations, where ${\cal F}$ consists
of even cycles, a diamond, $K_4$,  and $\overline C_6$ (see
Figure 1).

\begin{figure}[H]
\centering
\unitlength 1.5mm 
\linethickness{0.6pt}
\ifx\plotpoint\undefined\newsavebox{\plotpoint}\fi 
\begin{picture}(80,20)(0,0)
\put(1,17){\line(1,0){10}}
\put(11,17){\line(0,-1){10}}
\put(11,7){\line(-1,0){10}}
\put(1,7){\line(0,1){10}}
\put(1,17){\line(1,-1){10}}
\qbezier(11,17)(24,0)(1,7)
\put(27,17){\line(1,0){10}}
\put(37,17){\line(-5,-3){5}}
\put(32,14){\line(-5,3){5}}
\put(27,17){\line(0,-1){10}}
\put(27,7){\line(0,-1){1}}
\put(30,9){\line(0,1){0}}
\put(27,6){\line(5,3){5}}
\put(32,9){\line(5,-3){5}}
\put(37,6){\line(0,1){11}}
\put(37,17){\line(0,1){0}}
\put(27,6){\line(1,0){10}}
\put(32,14){\line(0,-1){5}}
\put(45,12){\line(0,1){0}}
\put(52,17){\line(-6,-5){6}}
\put(52,17){\line(6,-5){6}}
\put(58,12){\line(-1,-3){2}}
\put(56,6){\line(-4,5){4}}
\put(52,11){\line(6,1){6}}
\put(52,15){\line(0,1){2}}
\put(52,15){\line(0,-1){4}}
\put(52,11){\line(-6,1){6}}
\put(46,12){\line(1,-3){2}}
\put(48,6){\line(1,0){8}}
\put(52,11){\line(-4,-5){4}}
\put(1,17){\circle*{0.8}}
\put(11,17){\circle*{0.8}}
\put(11,7){\circle*{0.8}}
\put(1,7){\circle*{0.8}}
\put(27,17){\circle*{0.8}}
\put(37,17){\circle*{0.8}}
\put(27,6){\circle*{0.8}}
\put(37,6){\circle*{0.8}}
\put(46,12){\circle*{0.8}}
\put(52,17){\circle*{0.8}}
\put(58,12){\circle*{0.8}}
\put(56,6){\circle*{0.8}}
\put(48,6){\circle*{0.8}}
\put(7,1){\makebox(0,0)[cc]{$K_4$}}
\put(32,1){\makebox(0,0)[cc]{$\overline{C_6}$}}
\put(52,1){\makebox(0,0)[cc]{$W_5$}}
\put(32,14){\circle*{0.8}}
\put(32,9){\circle*{0.8}}
\put(52,11){\circle*{0.8}}

\put(73,16){\line(-6,-5){6}}
\put(67,11){\line(6,-5){6}}
\put(73,6){\line(6,5){6}}
\put(79,11){\line(-6,5){6}}
\put(73,16){\line(0,-1){10}}
\put(73,16){\circle*{0.7}}
\put(67,11){\circle*{0.7}}
\put(73,6){\circle*{0.7}}
\put(79,11){\circle*{0.7}}
\put(73,1){\makebox(0,0)[cc]{a diamond}}
\put(43,-5){\makebox(0,0)[cc]{Fig. 1. The four graphs}}
\end{picture}
\end{figure}

\vspace{0.2cm}
\section{Some basic results}
\label{sec:some}

We begin with some notions and notations.
A path $P$ is \emph{odd} or \emph{even}, if its length is odd or even, respectively.
A \emph{$uv$-path} is a path with ends $u$ and $v$. Let $G$ be a graph.
A $k$-vertex of $G$ is a vertex of degree $k$.
The set of neighbours of a vertex $v$ in $G$ is denoted by $N_G(v)$.
An \emph{even subdivision} of $G$ at an edge $e$ is a graph obtained from $G$ by replacing $e$ by an odd path $P_e$ with length at least three.
An \emph{odd expansion} of $G$ at a vertex $v$ of $G$ is a graph obtained from $G$
by the following four operations:

\noindent(\emph{i}) splitting $v$ into two vertices $v'$ and $v''$,

\noindent(\emph{ii}) adding an even path $P_v$ with length at least two which connects $v'$ and $v''$,

\noindent(\emph{iii}) distributing the edges of $G$ incident with $v$ among $v'$ and $v''$ such that $v'$ and $v''$ each has at least one neighbour in $V(G-v)$, and

\noindent(\emph{iv}) adding some edges joining $v'$ and $v''$ or not.

\noindent If the edges which join $v'$ and $v''$ are added,
then the odd expansion of $G$ is called an \emph{odd A-expansion};
otherwise, the odd expansion of $G$ is called an \emph{odd L-expansion}.

A graph $G$ is \emph{outerplanar} if it has a planar embedding in
which all vertices lie on the boundary of its outer face. An
outerplanar graph equipped with such an embedding is called an
\emph{outerplane graph}. Let $K$ be a 2-vertex cut of a graph $G$,
and $X$ the vertex set of a component of $G-K$. The subgraph of
$G$ induced by $X\cup K$ is called a $K$-\emph{component} of $G$.
Modify a $K$-component by adding a new edge (possible a multiple
edge) joining the two vertices of $K$. We refer to the modified
$K$-components as \emph{marked} $K$-\emph{components}.

Let $Y$ be a subset of $V(G)$. We use $\nabla(Y)$ to denote the \emph{edge cut} of $G$,
whose edges have one end in $Y$ and the other in $\overline{Y}$.
The graph $G\{Y\}$ and $G\{\overline{Y}\}$ are obtained from $G$ by contracting $\overline{Y}$ and $Y$ to a single vertex, respectively. 

The following lemmas will be used to obtain the main results of this paper.

\begin{lemma}(\cite{Bondy2008})\label{property-5}
In a loopless 3-connected plane graph $G$, for any vertex $v$,
the neighbours of $v$ lie on a common cycle
which is the boundary of the face of $G-v$ where the vertex $v$ is situated.   \end{lemma}

\begin{lemma}(\cite{Bondy2008})\label{property-6}
Every simple planar graph has a vertex of degree at most five.   \end{lemma}



 From exercise 11.2.7 in \cite{Bondy2008}, we have the following Lemma \ref{property-7'}.

\begin{lemma} \label{property-7'}
Every simple 2-connected outerplanar graph has two nonadjacent vertices of degree two. \end{lemma}


\begin{lemma}(\cite{Bondy2008})\label{property-2}
Let $G$ be a 2-connected graph, and $K$ a 2-vertex cut of $G$.
Then the marked $K$-components are also 2-connected.  \end{lemma}

\begin{lemma}\label{property-4}
Let $e$ be an edge of a graph $G$, and $G'$ an even subdivision of $G$ at $e$.
Then $G'$ is cycle-nice if and only if $G$ is cycle-nice.   \end{lemma}

\begin{proof} Let $u$ and $v$ be the two ends of $e$.
Since $G'$ is an even subdivision of $G$ at $e$,
$G'$ has an odd $uv$-path $P_e$
which is used to replace $e$ and whose internal vertices are 2-vertices of $G'$.

Suppose first that $G'$ is a cycle-nice graph. Let $C$ be an even cycle of $G$.
If $e\in E(C)$, let $C':=(C-e)\cup P_e$.
Then $C'$ is an even cycle of $G'$ such that $G-V(C)=G'-V(C')$.
Since $G'$ is  cycle-nice, $G'-V(C')$ has a perfect matching $M$.
Then $M$ is also a perfect matching of $G-V(C)$.
If $e\notin E(C)$, then $C$ is an even cycle of $G'$.
Let $M$ be a perfect matching of $G'-V(C)$.
If $M\cap E(P_e)$ is a perfect matching of $P_e$, let $M'=(M\setminus E(P_e))\cup\{e\}$; otherwise, let $M'=M\setminus E(P_e)$.
Then $M'$ is a perfect matching of $G-V(C)$.
Hence, $G$ is cycle-nice.

Conversely, suppose that $G$ is cycle-nice. Let $C'$ be an even cycle of $G'$.
If $E(P_e)\cap E(C')\neq\emptyset$, then $P_e$ is a segment of $C'$.
Let $C$ be the cycle obtained from $C'$ by replacing $P_e$ by the edge $e$.
Then $C$ is an even cycle of $G$ such that $G-V(C)=G'-V(C')$.
Let $M'$ be a perfect matching of $G-V(C)$.
Then $M'$ is also a perfect matching of $G'-V(C')$.
If $E(P_e)\cap E(C)=\emptyset$, then $C'$ is an even cycle of $G$.
Let $M'$ be a perfect matching of $G-V(C')$,
and let $M_1'$ be a perfect matching of $P_e$.
If $e\in M'$, then $(M'\setminus \{e\})\cup M'_1$ is a perfect matching of $G'-V(C')$;
if $e\notin M'$, then $M'\cup (M_1'\triangle E(P_e))$ is a perfect matching of $G'-V(C')$.
Consequently, $G'$ is cycle-nice. The lemma follows.
\end{proof}

\begin{lemma}\label{property-8}
Let $G$ be a 2-connected claw-free graph with a 2-vertex cut $\{u,v\}$. Then

($\romannumeral1$) $G-\{u,v\}$ has exactly two components $G_1$ and $G_2$.

Furthermore, suppose that $G$ is cycle-nice.
Let $G_i'$ be the underlying simple graph of the graph
which is obtained from the $\{u,v\}$-component of $G$
containing $G_i$ by deleting the possible edges of $G$
connecting $u$ and $v$. Then

($\romannumeral2$) at least one of $G_1'$ and $G_2'$ is a path (suppose that $G_2'$ is a path in the following statements);

($\romannumeral3$) if $G'_2$ is an odd path, then the two marked  $\{u,v\}$-components of $G$ are cycle-nice;

($\romannumeral4$) if $G'_2$ is an even path, then $G\{V(G_1)\}$ and $G\{V(G_2)\}$ are cycle-nice.
In particular, $uv\notin E(G)$ when neither $G'_1$ nor $G'_2$ is a path of length two.
\end{lemma}

\begin{proof} Suppose that $G_1$, $G_2,$ \ldots, $G_t$ are the components of $G-\{u,v\}$, where $t\geq 2$.
Since $G$ is claw-free, for each vertex in $\{u,v\}$,
its neighbours lie in at most two components of $G-\{u,v\}$.
If $t\geq 3$, then there exists a component, say $G_1$, such that only one of $u$ and $v$, say $u$, has neighbours in $G_1$.
Then $u$ is a cut vertex of $G$, a contradiction. So, $t=2$. ($\romannumeral1$) follows.

Now, suppose that $G$ is cycle-nice. To show ($\romannumeral2$), we first prove the following claim.

\noindent{\bf Claim.} If $|V(G_i)|\geq 2$ and $G_i'$ is not a path,
then in $G_i'$ there are two $uv$-paths which have different parity, $i=1$ or 2.

For convenience, suppose that $i=1$. Since $G$ is 2-connected and $|V(G_1)|\geq 2$, there are two nonadjacent edges $uu'$ and $vv'$ with $u', v'\in V(G_1)$.
Since $G_1$ is connected, there is a path  $P'$ in $G_1$ connecting $u'$ and $v'$.
Let $P_1=uu'P'v'v$. Then $P_1$ is a  $uv$-path in $G_1'$ with length at least three.
Suppose that $P_1=x_0x_1\cdots x_sx_{s+1}$ is a longest such path in $G_1'$,
where $s\geq 2$, $x_0=u$ and $x_{s+1}=v$.
We will find another $uv$-path $P_1'$ in $G_1'$ such that $P_1$ and $P_1'$ have different parity.

When $V(G_1')\neq V(P_1)$, since $G_1$ is connected, there is a vertex $x$ of $G_1-V(P_1)$ such that
$xx_j\in E(G_1)$, $j\in\{1,2,\ldots,s\}$.
Since $G$ is claw-free, at least one of $xx_{j-1}$, $xx_{j+1}$ and $x_{j-1}x_{j+1}$
is an edge of $G_1'$.
Since $P_1$ is a longest $uv$-path, we have $xx_{j-1}$, $xx_{j+1}\notin E(G_1')$,
and so, $x_{j-1}x_{j+1}\in E(G_1')$. Let $P_1'=x_0P_1x_{j-1}x_{j+1}P_1x_{s+1}$.
Then $P_1'$ is the desired path.

When $V(G_1')=V(P_1)$, since $G'_1$ is not a path, there is an
edge $x_jx_k$ in $E(G'_1)\setminus E(P_1)$ with
$j,k\in\{0,1,2,\ldots,s+1\}$, $k> j+1$ and $\{j,k\}\neq
\{0,s+1\}$. Suppose first that for any edge $x_ix_{i'}$ in
$E(G'_1)\setminus E(P_1)$, $i$ and $i'$ have different parity. So
$j$ and $k$ have different parity and $k>j+2$. If $j\geq 1$,
since $j-1$, $j+1$ and $k$ have the same parity, we have
$x_{j-1}x_{j+1}$, $x_{j-1}x_k$, $x_{j+1}x_k \notin E(G_1)$. Then the
subgraph induced by $x_j$ and its three neighbors
$x_{j-1}$, $x_{j+1}$ and $x_k$ contains a claw, a contradiction. If $j=0$,
then $k\geq 3$, $k\neq s+1$ and $k$ is odd. Then $x_1x_k\notin
E(G_1)$. In this case, the subgraph induced by $x_0$ and its three
neighbours $x_1$, $x_k$ and one in $V(G_2)$ contains a claw, a
contradiction. Thus, we may suppose that $j$ and $k$ have the same
parity. Let $P_1'=x_0P_1x_jx_kP_1x_{s+1}$. Then $P_1'$ is the
desired path. The claim follows.

Now we continue the proof of statement ($\romannumeral2$).
If $G_1$ or $G_2$ is trivial, then ($\romannumeral2$) holds trivially.
Suppose in the following that $|V(G_i)|\geq2$ for $i=1,2$,
and neither $G_1'$ nor $G_2'$ is a path.
Since $|V(G)|$ is even, $|V(G_1)|$ and $|V(G_2)|$ have the same parity.
From the above claim, when $G_1$ and $G_2$ are odd components,
we may suppose that $P_i$ is an odd $uv$-path of $G_i'$, $i=1,2$;
when $G_1$ and $G_2$ are even components, we may suppose that $P_i$ is an even $uv$-path of $G_i'$, $i=1,2$.
Let $C=P_1\cup P_2$. Then $C$ is an even cycle of $G$.
Since $|V(G_i)\setminus V(C)|$ is odd, $G-V(C)$ has no perfect matching, a contradiction to the assumption that $G$ is cycle-nice.
So, at least one of $G'_1$ and $G'_2$ is a path. This proves ($\romannumeral2$).


($\romannumeral3$)
Since $G$ is 2-connected, by Lemma \ref{property-2},
the marked $\{u,v\}$-components of $G$, say $H_1$ and $H_2$, are 2-connected.
Suppose that $V(H_i)=V(G_i)\cup\{u,v\}$, $i=1,2$.
Since $G'_2$ is an odd path, $|V(H_1)|$ and $|V(H_2)|$ are even.
So the underlying simple graph of $H_2$ is an even cycle.
Thus $H_2$ is cycle-nice. Replace the odd path $G'_2$ of $G$ by an edge,
which connects $u$ and $v$. The resulting graph is $H_1$.
Since $G$ is cycle-nice, Lemma \ref{property-4} implies that $H_1$ is cycle-nice. ($\romannumeral3$) follows.

($\romannumeral4$) Since $G_2'$ is an even path, $G_1$ and $G_2$ are odd components.
We first suppose that neither $G_1'$ nor $G_2'$ is a path of length two, that is,
both $G_1$ and $G_2$ are nontrivial. We will show that $uv\notin E(G)$.
Recall that $G'_2$ is an even path. Then $G'_2$ has at least five vertices.
Suppose, to the contrary, that $uv\in E(G)$.
If $G'_1$ is a path, then it also has  at least five vertices.
This implies that $G$ contains a claw formed by $u$ and
its three neighbours (one in $G_1$, one in $G_2$, and $v$), a contradiction.
Thus, $G'_1$ is not a path. Since $G$ is 2-connected and $|V(G_1)|\geq 3$,
in $G_1$ there are a neighbour $u_1$ of $u$ and a neighbour $v_1$ of $v$ such that $u_1\neq v_1$.
Since $G$ is claw-free, considering $u$ and its three neighbours $u_1$, $v$ and one in $G'_2$, we have $u_1v\in E(G)$,
and considering $v$ and its three neighbours $v_1$, $u_1$ and one in $G'_2$,
we have $u_1v_1\in E(G)$. Let $C$ be the cycle $uu_1v_1vu$.
Then $G_2$ is an odd component of $G-V(C)$, and so,
$G-V(C)$ has no perfect matching, a contradiction to the assumption that $G$ is cycle-nice. Hence, $uv\notin E(G)$.

Recall that $G\{V(G_i)\}$ is the graph obtained from $G$ by contracting $\overline {V({G_i})}$ to a vertex $x_i$, $i=1,2$. Since
the underlying simple graph of $G\{V(G_2)\}$ is an even cycle,
$G\{V(G_2)\}$ is cycle-nice. To show that $G\{V(G_1)\}$ is
cycle-nice, let $C_1$ be an even cycle of $G\{V(G_1)\}$. If
$x_1\in V(C_1)$, let $e_1$ and $e_2$ be the two edges incident
with $x_1$ in $C_1$. If $e_1\in \nabla(u)$ and $e_2\in\nabla(v)$,
let $C=(C_1-x_1)\cup\{e_1,e_2\}\cup G'_2$. Then $C$ is an even
cycle of $G$ and $G\{V(G_1)\}-V(C_1)=G-V(C)$. Since $G$ is
cycle-nice, $G-V(C)$ has a perfect matching $M$, which is also a
perfect matching of $G\{V(G_1)\}-V(C_1)$.
If  $e_1$ and $e_2$ belong to  one of $\nabla(u)$ and $\nabla(v)$,  then $C_1$ is an even cycle of $G$.  If  $x_1\notin V(C_1)$, then $C_1$ is also an even cycle of $G$. Let $M$ be a
perfect matching of $G-V(C_1)$. Since $|V(G_2)|$ is odd, $u$ and $v$
are matched to different components of $G-\{u,v\}$ under any
perfect matching of $G$.
Thus, $M\cap E(G\{V(G_1)\})$ is a perfect matching of $G\{V(G_1)\}-V(C_1)$.
Consequently, $G\{V(G_1)\}$ is cycle-nice. ($\romannumeral4$) follows.
\end{proof}

\begin{lemma}\label{property-4.3} Let $G$ be a 2-connected graph and let $G'$ be an even subdivision or an odd expansion of $G$.
If $G'$ is claw-free, then $G$ is claw-free.
\end{lemma}

\begin{proof}
Suppose first that $G'$ is an even subdivision of $G$ at an edge $e$, that is,
$G'$ is obtained from $G$ by replacing $e$ by an odd path $P_{e}$.
Let $u$ and $v$ be the two ends of $e$. If $G$ has a claw,
then the only possibility is that the center of the claw is $u$ or $v$.
Since $G'$ is claw-free, $N_G(u)\setminus \{v\}$ and  $N_G(v)\setminus \{u\}$ are cliques of $G'$, which are also cliques of $G$.
So $G$ has no claw with center $u$ or $v$. Thus, $G$ is claw-free.

We next suppose that $G'$ is an odd expansion of $G$ at a vertex $u$.
Let $u'$ and $u''$ be the two split vertices of $G'$.
Recall that $P_{u}$ is the $u'u''$-path of $G'$. Since $G'$ is claw-free,
the neighbours of $u'$ in $V(G-u)$ form a clique of $G$.
Otherwise, the subgraph of $G'$ induced by $u'$ and its three neighbours,
two nonadjacent neighbours in $V(G-u)$ and one in $P_{u}$,
contains a claw, a contradiction.
Similarly, the subgraph induced by the neighbours of $u''$ in $V(G-u)$
also contains a clique of $G$.
This implies that  $G$ is claw-free.
\end{proof}

Let $G$ be a 2-connected claw-free cycle-nice graph. If $G$ has a 2-vertex cut $K$,
by Lemma \ref{property-8}($\romannumeral1$), $G-K$ has only two components $G_1$ and $G_2$.
By Lemma \ref{property-8} and Lemma \ref{property-4.3}, either the two marked $K$-components of $G$ or both $G\{V(G_1)\}$ and $G\{V(G_2)\}$,
which are 2-connected, are claw-free and cycle-nice.
If the smaller 2-connected claw-free cycle-nice graph still has a 2-vertex cut,
we will repeat this procedure until we obtain a 3-connected claw-free cycle-nice graph,
whose underlying simple graph is $K_2$ or a 3-connected claw-free cycle-nice graph.

\section{3-connected graphs}
\label{sec:3-connected}

\begin{theorem}\label{property-9} Let $G$ be a simple 3-connected claw-free plane graph.
Then $G$ is cycle-nice if and only if $G$ is $K_4$, $W_5$ or $\overline C_6$. \end{theorem}

\begin{proof} It is easy to check that $K_4$, $W_5$ and $\overline C_6$ are cycle-nice.
Now, we assume that $G$ is cycle-nice.
Lemma \ref{property-5} implies that for any vertex $x$ of $G$, the neighbours of $x$ lie on a common cycle, denoted by $C_x$,
which is the boundary of the face of $G-x$ in which $x$ is situated.
If $C_x$ is an even cycle, then $x$ is an isolated vertex of $G-V(C)$,
a contradiction to the assumption that $G$ is cycle-nice.
Thus, for any vertex $x$ of $G$, $C_x$ is an odd cycle.
By Lemma \ref{property-6}, $\delta(G)\leq 5$.
Since $G$ is 3-connected, we have $3\leq\delta(G)\leq 5$.
Let $u$ be a vertex of $G$ such that $d(u)=\delta(G)$.
Let $C_u=u_0u_1\cdots u_su_0$. Then $s$ is even and $s\geq 2$.
We first prove the following claim.
All subscripts are taken modulo $s+1$ in the following.

\noindent{\bf Claim 1.} $V(G)=V(C_u)\cup\{u\}$.

Suppose, to the contrary, that there is a vertex $v\in V(G)\setminus (V(C_u)\cup\{u\})$.
We may further suppose that $vu_i\in E(G)$ for some $i\in\{0,1,\ldots,s\}$.
If $u_{i-1}v$ or $u_{i+1}v$ is an edge of $G$,
let $C'=C_u-u_{i-1}u_i+u_{i-1}vu_i$ or $C'=C_u-u_iu_{i+1}+u_ivu_{i+1}$.
Then $C'$ is an even cycle of $G$ such that $u$ is an isolated vertex of $G-V(C')$,  a contradiction.
Thus $u_{i-1}v$, $u_{i+1}v\notin E(G)$.
Since $G$ is claw-free, we have $u_{i-1}u_{i+1}\in E(G)$.
Furthermore, when $s>2$, $u_{i-1}u_{i+1}$ lies in the exterior of $C_u$,
and $v$ lies in the interior of the cycle $u_{i-1}u_iu_{i+1}u_{i-1}$.
If $uu_i\notin E(G)$, since $\delta(G)\geq 3$, we have $s>2$.
Then $G-\{u_{i-1},u_{i+1}\}$ has at least two components, one contains $v$ and the other contains $u$,
a contradiction to the assumption that $G$ is 3-connected. Thus, $uu_i\in E(G)$.
Consider the vertex $u_i$ and its three neighbours $u$, $u_{i-1}$ and $v$.
Since $G$ is claw-free and $vu_{i-1}$, $vu\notin E(G)$, we have $uu_{i-1}\in E(G)$.
When consider $u_i$ and its three neighbours $u$, $u_{i+1}$ and $v$, we have  $uu_{i+1}\in E(G)$.
Recall that $C_{u_i}$ is the facial cycle of the face of $G-u_i$ which contains all neighbours of $u_i$.
Then we have $u,v,u_{i-1},u_{i+1}\in V(C_{u_i})$.
Let $P_1$ and $P_2$ be the two segments of $C_{u_i}$ such that $u_{i-1}P_1vP_2u_{i+1}uu_{i-1}=C_{u_i}$.
Since $C_{u_i}$ is an odd cycle, one of $P_1$ and $P_2$ is an odd path,
and the other is an even path.
Suppose, without loss of generality, that $P_1$ is an odd path.
Let $C'=(C_u-u_{i-1}u_i)+u_{i-1}P_1vu_i$.
Then $C'$ is an even cycle of $G$ such that $u$ is an isolated vertex of $G-V(C')$, a contradiction. Claim 1 follows. \hfill$\Box$

By Claim 1, $G-u$ is a 2-connected outerplane graph.
We now let $u$ be situated in the outer face of $G-u$.
By Lemma \ref{property-7'}, $G-u$ has a vertex $x$ of degree two.
Since $\delta(G)\geq 3$, we have $xu\in E(G)$ and $d_G(x)=3$.
Recall that $d_G(u)=\delta(G)$. We have $d_G(u)=3$.
Let $N_G(u)=\{u_i,u_j,u_l\}$, $0\leq j<l<i\leq s$.
Then $u_i$, $u_j$ and $u_l$ effect a partition of $C_u$ into three edge-disjoint paths $P_1$, $P_2$ and $P_3$,
which connect $u_l$ and $u_i$, $u_i$ and $u_j$, $u_j$ and $u_l$, respectively.

Since $G$ is claw-free, at least one of $u_iu_j$, $u_ju_l$ and $u_lu_i$ is an edge of $G$.
Suppose, without loss of generality, that $u_ju_l\in E(G)$.
If $P_3\neq u_ju_l$, from the fact that $G-u$ is an outerplane graph,
we find that $\{u_j,u_l\}$ is a 2-vertex cut of $G$,
a contradiction to the assumption that $G$ is 3-connected.
Thus, $P_3=u_ju_l$.
So we may suppose, without loss of generality,
that $N(u)=\{u_0,u_1,u_i\}$ for some $i\in\{2,3,\ldots,s\}$.
Then $j=0$ and $l=1$, and so, $P_1$ is a $u_1u_i$-path,
$P_2$ is a $u_iu_0$-path, and $P_1$ and $P_2$ have the same parity.
Since $G-u$ is a simple outerplane graph,
there are at least two nonadjacent 2-vertices in $G-u$ by Lemma \ref{property-7'}.
Since $\delta(G)\geq3$, each 2-vertex of $G-u$ is a neighbour of $u$ in $G$.
Thus $G-u$ has at most three 2-vertices ($u_0, u_1$, or $u_i$) and $d_{G-u}(u_i)=2$.

Consider $u_i$ and its three neighbours.
For distinct indices $i'$ and $i''$,
if $\{i',i''\}\subseteq\{1,2,3,\ldots,i-1\}$ or   $\{i',i''\}\subseteq\{i+1,i+2,\dots,s,0\}$ satisfying $|i'-i''|\geq2$,
then $u_{i'}u_{i''}\notin E(G)$.
Otherwise, since $G-u$ is a simple outerplane graph,
$\{u_{i'}, u_{i''}\}$ is a 2-vertex cut of $G$,
a contradiction to the 3-connectivity of $G$.
Recall that $\delta(G)\geq3$ and $d_G(u_i)=3$.
If $k\in \{2,3,\ldots,i-1\}$, then $\emptyset\neq N_G(u_k)\setminus\{u_{k-1},u_{k+1}\}\subseteq V(P_2)\setminus \{u_i\}$;
if  $k\in\{i+1,i+2,\dots,s\}$,
then $\emptyset\neq N_G(u_k)\setminus\{u_{k-1},u_{k+1}\}\subseteq V(P_1)\setminus\{u_i\}$.
We now show the following claim,
which implies that the number of vertices of $C_u$ is at most five.

\noindent{\bf Claim 2.}  $s\leq 4$.

Suppose, to the contrary,  that $s>4$. Then $s\geq 6$. If $i=2$ or $i=s$, by symmetry, we need only consider the case where $i=2$.
Since $G-u$ has two nonadjacent 2-vertices, we have  $d_{G-u}(u_0)=d_{G-u}(u_2)=2$.
Then, for each $j\in\{3,4,\ldots,s\}$, $u_j$ is adjacent to $u_1$,
which is the only vertex in $V(P_1)\setminus\{u_i\}$.
Since $G-u$ is an outerplane graph, $u_1$ and its three neighbours $u_0$,
$u_2$ and $u_4$ form a claw in $G$, a contradiction.
Thus, $i\neq 2$ and $i\neq s$.
This implies that the lengths of $P_1$ and $P_2$ are at least two.

If one of $P_1$ and $P_2$ has length two, say $P_1$, then $i=3$ and $P_1=u_1u_2u_3$. Since $P_1$ and $P_2$ have the same parity and $s\geq 6$,  $P_2$ has at least five vertices.
If $u_1u_s\in E(G)$, let $C=uu_3u_4\dots u_su_1u$.
Then $C$ is an even cycle of $G$ such that $u_0$ and $u_2$ are two isolated vertices of $G-V(C)$.
Thus $G-V(C)$ has no perfect matching.
This contradiction implies that $u_1u_s\notin E(G)$.
Then $u_s$ is adjacent to $u_2$, the only vertex in $V(P_1)\setminus\{u_i,u_1\}$.
Consequently, all internal vertices of $P_2$ are adjacent to $u_2$.
Then $u_2$ and its three neighbours $u_1$, $u_3$ and $u_s$ form a claw, a contradiction.
Hence, $P_1$ and $P_2$ have length at least three.

We now consider $u_i$ and its three neighbours $u_{i-1}$, $u_{i+1}$ and $u$.
Since $uu_{i-1}$, $uu_{i+1}\notin E(G)$, we have $u_{i-1}u_{i+1}\in E(G)$.
We assert that $u_{i-2}u_{i+2}\in E(G)$.
In fact, if $u_{i-1}u_{i+2}\in E(G)$,
by considering $u_{i-1}$ and its three neighbours $u_{i-2}$,  $u_i$ and $u_{i+2}$,
we have $u_{i-2}u_{i+2}\in E(G)$.
Similarly, if $u_{i-2}u_{i+1}\in E(G)$, we have $u_{i-2}u_{i+2}\in E(G)$.
For the case where $u_{i-1}u_{i+2}\notin E(G)$ and $u_{i-2}u_{i+1}\notin E(G)$,
we have $N_G(u_{i-2})\setminus\{u_{i-1},u_{i-3}\}\subseteq \{u_{i+2}, \ldots, u_s, u_0\}$
and $N_G(u_{i+2})\setminus\{u_{i+1},u_{i+3}\}\subseteq \{u_1,\ldots,u_{i-2}\}$.
Recall that $\delta(G)\geq3$.
If $u_{i-2}u_{i+2}\notin E(G)$,
then there is an index $j\in \{1, 2, \ldots, {i-3}\}$ such that $u_{i+2}u_{j}\in E(G)$.
Recall that $G-u$ is a outerplane graph, whose boundary of outer face is $C_u$.
Then $u_{i-2}$ is a 2-vertex of $G$, a contradiction.
The assertion follows.

If $P_1$  and $P_2$ are even, let $C=u_{i+2}u_{i-2}u_{i-1}u_iuu_1u_0u_s\dots u_{i+2}$;
if $P_1$  and $P_2$ are odd, let $C=u_{i+2}u_{i-2}u_{i-1}u_iuu_0u_s\dots u_{i+2}$.
Then $C$ is an even cycle of $G$ such that $u_{i+1}$ is an isolated vertex of $G-V(C)$, a  contradiction. Claim 2 follows.\hfill$\Box$

Recall that $s\geq 2$. If $s=2$, $G$ is $K_4$.
If $s=4$, we have $C_u=u_0u_1u_2u_3u_4u_0$.
If $i=2$, then  $d_{G-u}(u_3)\geq3$, $d_{G-u}(u_4)\geq3$, and so,
$u_1u_3,u_1u_4\in  E(G)$.
Hence, $G$ is $W_5$. By symmetry, if $i=4$, $G$ is also $W_5$.
If $i=3$, then $d_{G-u}(u_2)\geq3$, $d_{G-u}(u_4)\geq3$.
If $u_4$ is adjacent to $u_1$, let $C=uu_1u_4u_3u$.
Then $G-V(C)$ has no perfect matching.
So $u_4u_1\notin E(G)$. By symmetry, we have $u_2u_0\notin E(G)$.
Consequently, $u_2u_4\in E(G)$. Hence, $G$ is $\overline C_6$.
The result follows.
\end{proof}

\section{2-connected graphs}
\label{sec:2-connected}

In this section, we use $G_{sim}$ to denote the underlying simple graph of $G$.
An edge of $G$ is \emph{admissible} if it is contained in a perfect matching of $G$.
A \emph{quasi-diamond} is an odd A-expansion of a cycle of length two (see Fig. 2).
Obviously, a diamond is a quasi-diamond. Note that a quasi-diamond is also an even subdivision of a diamond at an edge whose one end vertex is a  2-vertex of the diamond.

\begin{figure}[H]
  \centering
\unitlength 1.7mm 
\linethickness{0.6pt}
\ifx\plotpoint\undefined\newsavebox{\plotpoint}\fi 
\begin{picture}(60,17)(0,0)
\put(25,15){\line(-6,-5){6}}
\put(19,10){\line(6,-5){6}}
\put(25,5){\line(0,1){9}}
\put(25,14){\line(0,1){1}}
\qbezier(25,15)(45,10)(25,5)
\put(19,10){\circle*{0.7}}
\put(25,15){\circle*{0.7}}
\put(25,5){\circle*{0.7}}
\put(31,13){\circle*{0.7}}
\put(35,10){\circle*{0.7}}
\put(31,7){\circle*{0.7}}
\put(28,1){\makebox(0,0)[cc]{Fig. 2. A quasi-diamond with six vertices}}
\end{picture}
\end{figure}

\begin{lemma}\label{property-4.1} Let $G$ be a 2-connected claw-free graph, and $G'$ an odd A-expansion of $G$.
If $G'$ is claw-free and cycle-nice,
then $G_{sim}$ is an even cycle or $K_2$,
and $G_{sim}'$ is a quasi-diamond.
\end{lemma}

\begin{proof}
Suppose that $G'$ is an odd A-expansion of $G$ at a vertex $v$,
and $v'$ and $v''$ are the two split vertices.
Then $v'v''\in E(G')$, and $P_v$ is the even $v'v''$-path.
Since $G'$ is cycle-nice, $v'v''$ does not lie in any even cycle of $G'$.
Suppose, to the contrary, that $G_{sim}$ is neither an even cycle nor $K_2$.
Since $G$ is 2-connected and $G_{sim}$ is not $K_2$, in $G$ the vertex $v$ has at least two neighbours.
According to the number of neighbours of $v$ in $G$,
we distinguish the following two cases.

\noindent{\bf Case 1.} $|N_G(v)|=2$.
Let $v_1$ and $v_2$ be the two  neighbours of $v$ such that $v'v_1, v''v_2\in E(G')$.
Let $P'$ be a shortest $v_1v_2$-path in $G_{sim}-v$,
and $C'$ the union of the two paths $v_1v'v''v_2$ and $P'$.
Then $C'$ is a  cycle of $G'$ containing $v'v''$.
Thus, $C'$ is an odd cycle, and $P'$ is an even path.
Recall that $G_{sim}$ is not an even cycle.
Since $P'$ is a shortest $v_1v_2$-path, $C'$ has no chord.
Therefore, there is a vertex $u\in V(G_{sim}-v)\setminus V(P')$
such that $u$ is adjacent to a vertex $u'$ of $P'$.
Since $C'$ has no chord, the two neighbours of $u'$ in $C'$ are not adjacent.
Since $G'$ is claw-free, at least one neighbour of $u'$ in $P'$ is adjacent to $u$, say $u''$.
Then the union of the two paths $C'-u'u''$ and $u'uu''$ is an even cycle of $G'$,
which contains $v'v''$, a contradiction.

\noindent{\bf Case 2.} $|N_G(v)|\geq 3$.
Since $G$ is claw-free, at least two neighbours, say $v_1$ and $v_2$,  of $v$ are adjacent in $G$.
If $v_1\in N_{G'}(v')$ and $v_2\in N_{G'}(v'')$,
then $v'v_1v_2v''v'$ is an even cycle of $G'$ containing $v'v''$, a contradiction.
Thus, we may suppose that $\{v_1,v_2\}\subseteq N_{G'}(v')$.
Let $v_3\notin \{v_1, v_2\}$ be a neighbour of $v$ adjacent to $v''$ in $G'$.
Since $G_{sim}$ is 2-connected, there exists a path $Q$ in $G_{sim}-v$ from $v_3$ to $\{v_1,v_2\}$.
Suppose that $Q$ is a shortest such path and  $v_1\in V(Q)$.
Then $v_2\notin V(Q)$. Thus, one of $v'v_1Qv_3v''v'$ and $v'v_2v_1Qv_3v''v'$ is an even cycle of $G'$ containing $v'v''$, a  contradiction.

From the above discussion, $G_{sim}$ is either an even cycle or $K_2$.
Recall that $G'$ is claw-free.
If $G_{sim}$ is an even cycle, then $P_v$ is an even path with 3 vertices.
Thus, $G_{sim}'$ is a quasi-diamond.
If $G_{sim}$ is $K_2$, then $P_v$ is an even path with at least 3 vertices.
Thus, $G_{sim}'$ is a quasi-diamond.
\end{proof}

\begin{lemma}\label{property-4.2}
Let $G$ be a 2-connected cycle-nice graph, and let $G'$ be an odd L-expansion of $G$. Then $G'$ is  cycle-nice.
\end{lemma}

\begin{proof} Suppose that $G'$ is an odd L-expansion of $G$ at a vertex $v$,
and $v'$ and $v''$ are the two split vertices.
Then $v'v''\notin E(G')$. Let $C'$ be an even cycle of $G'$.
If $|V(C')\cap\{v',v''\}|=0$ or $|V(C')\cap\{v',v''\}|=1$,
then $C'$ is an even cycle of $G$.
Since $G$ is cycle-nice, $G-V(C')$ has a perfect matching $M$.
If $|V(C')\cap\{v',v''\}|=0$, then $M$ has an edge $e$ which is incident with $v$.
Suppose that in $G'$ the edge $e$ is incident with $v'$.
If $|V(C')\cap\{v',v''\}|=1$, then suppose that $v'\in V(C')$.
Recall that $P_v$ is an even $v'v''$-path of $G'$.
So $P_v-v'$ has a perfect matching $M'$.
In both cases $M\cup M'$ is a perfect matching of $G'-V(C')$.

If $|V(C')\cap\{v',v''\}|=2$, then $P_v$ is a segment of $C'$.
Let $C$ be a cycle obtained from $C'$ by contracting $P_v$ to a vertex $v$.
Then $C$ is an even cycle of $G$, and so, $G-V(C)$ has a perfect matching $M$.
Since $G-V(C)=G'-V(C')$, $M$ is also a perfect matching of $G'-V(C')$.

From the above discussion, we conclude that $G'$ is cycle-nice.
\end{proof}

\begin{lemma}\label{property-4.4} If $G'$ is an even subdivision or an odd expansion of $W_5$ possibly with multiple edges, then $G'$ has a claw.
\end{lemma}

\begin{proof} Let $G$ be a graph such that $G_{sim}=W_5$.
Suppose that the cycle of $G_{sim}$ of length five is $C=u_1u_2u_3u_4u_5u_1$.
Let $u$ be the vertex of $G$ not in $V(C)$.
Then $uu_i\in E(G)$, $1\leq i\leq 5$.

Suppose that $G'$ is an even subdivision of $G$ at an edge $e$, that is,
$G'$ is obtained from $G$ by replacing $e$ by an odd path $P_e$
which has at least four vertices.
Note that $e$ has at least one end in $C$.
Suppose that $u_1$ is one end of $e$.
If the other end of $e$ is $u$, then $e=uu_1$.
Since $u_{2}u_{5}\notin E(G')$, the subgraph of $G'$ induced by $u$ and its three neighbours $u_{5}$, $u_{2}$,
and one in $P_e$ contains a claw.
If the other end of $e$ is $u_{2}$, then $e=u_{1}u_2$.
If $e$ is a single edge, then the subgraph of $G'$ induced by
$u$ and its three neighbours $u_{1}$, $u_2$ and $u_{4}$ contains a claw.
If $e$ is a multiple edge of $G$, then $u_{1}$ and $u_2$ are adjacent in $G'$.
Since $u_{1}u_{3}\notin E(G')$, the subgraph of $G'$ induced by
$u_2$ and its three neighbours $u_{1}$, $u_{3}$, and one in $P_e$ contains a claw.
Thus, if $G'$ is obtained from $G$ by an even subdivision of $G$,
then $G'$ has a claw.

We next suppose that $G'$ is an odd expansion of $G$ at a vertex $v$,
and $v'$ and $v''$ are the two split vertices.
Then $P_v$ is the even $v'v''$-path,
and $v'$ and $v''$ are adjacent to at least one vertex in $N_G(v)$, respectively.
Note that $v'v''$ is a possible edge of $G'$.
When $v=u$, since $u$ has five neighbours in $C$, one of $v'$ and $v''$,
say $v'$, has at least three neighbours in $C$.
Then $v'$ has two nonadjacent neighbours in $C$, say $u_1$ and $u_{3}$.
It follows that the subgraph of $G'$ induced by $v'$ and its three neighbours
$u_1$, $u_{3}$, and one in $P_v$ contains a claw.
When $v\neq u$, we may suppose that $v= u_2$.
Since $N_G(u_2)=\{u,u_{1},u_{3}\}$,
one of $v'$ and $v''$ has two neighbours in $N_G(u_2)$.
Suppose, without loss of generality, that $\{u,u_{1}\}\subseteq N(v')$ and $u_{3}\in N(v'')$.
If $u_{3}$ is not adjacent to $v'$ in $G'$,
then the subgraph of $G'$ induced by $\{u,v',u_{3},u_{5}\}$ contains a claw.
If $u_{3}$ is adjacent to $v'$ in $G'$,
then the subgraph of $G'$ induced by $v'$ and its three neighbours $u_{1}$,
$u_{3}$, and one in $P_v$ contains a claw. The lemma follows.
\end{proof}

Combining Lemma \ref{property-4.3} and Lemma \ref{property-4.4},
we see that a 2-connected claw-free graph cannot be obtained from $W_5$ by a sequence of edge subdivisions and vertex expansions.

\begin{theorem} Suppose that $G$ is a 2-connected claw-free plane graph possibly with multiple edges,
and $G_{sim}$ is not $W_5$.
Then $G$ is cycle-nice if and only if there exists a sequence
$(G_1, G_2, \ldots, G_r)$ of graphs such that
($\romannumeral1$) $G_1$ is an even cycle, a diamond, $K_4$,
or $\overline C_6$, and $G_r=G$,
($\romannumeral2$) $G_i$ is an even subdivision or an odd L-expansion of $G_{i-1}$,
or $G_i$ is obtained from $G_{i-1}$ by replacing some admissible edges of $G_{i-1}$ by some multiple edges.
\end{theorem}

\begin{proof} To prove the sufficiency, we show, by induction,
that each  $G_{i}$  ($1\leq i\leq r$) is a 2-connected  cycle-nice graph.
Since an even cycle, a diamond, $K_4$ and $\overline C_6$ are 2-connected
and cycle-nice, $G_1$ is a 2-connected cycle-nice graph.

Inductively, suppose that $2\leq i\leq r$ and $G_{i-1}$ is a 2-connected cycle-nice graph.
We consider the following possibilities.

$\bullet$ $G_{i}$ is an even subdivision or an odd L-expansion of $G_{i-1}$. By Lemma \ref{property-4} and Lemma \ref{property-4.2},
$G_{i}$ is cycle-nice.
Since an even subdivision and an L-expansion of a 2-connected graph are also 2-connected, $G_{i}$ is 2-connected.

$\bullet$ $G_i$ is obtained from $G_{i-1}$
by replacing an admissible edge $e$ of $G_{i-1}$
by a set of multiple edges $E_e$.
Then $G_i$ is 2-connected. Let $C$ be an even cycle of $G_i$.
If $C$ contains no edge of $E_e$, then $C$ is an even cycle of $G_{i-1}$.
Since $G_{i-1}$ is cycle-nice, $G_{i-1}-V(C)$ has a perfect matching,
which is also a perfect matching of $G_{i}-V(C)$.
So $C$ is a nice-cycle of $G_i$.
Then we consider the situation that $C$ contains an edge $e'$ of $E_e$.
If the length of $C$ is two, then $V(C)=V(e)$.
Since $e$ is an admissible edge of $G_{i-1}$, $G_{i}-V(C)=G_{i-1}-V(e)$
has a perfect matching, and so, $C$ is a nice cycle of $G_i$.
If the length of $C$ is at least four,
let $C'$ be the cycle obtained from $C$ by replacing the edge $e'$ by the edge $e$.
Then $C'$ is an even cycle of $G_{i-1}$ with
$V(C)=V(C')$, and so, $G_{i}-V(C)=G_{i-1}-V(C')$ has a perfect matching.
This implies that $C$ is a nice cycle of $G_i$.
Consequently, $G_i$ is a 2-connected cycle-nice graph, as desired.

Now, we give a proof of necessity by induction on the number of vertices of $G$.
From the assumption, we see that the underlying simple graph $G_{sim}$ of $G$ is a
2-connected claw-free cycle-nice plane graph.
If $G$ is 3-connected, then $G_{sim}$ is 3-connected.
By Theorem \ref{property-9} and the assumption, $G_{sim}$ is $K_4$ or $\overline C_6$.
If $G$ is simple, then $G_1=G$.
If $G$ is not simple, since $G$ is cycle-nice,
each cycle of $G$ of length two is a nice cycle.
Then each member of multiple edges of $G$ is admissible.
So $G$ is obtained from $G_{sim}$ by replacing some admissible edges of $G_{sim}$ by some multiple edges.
By setting $G_1=G_{sim}$ and $G_2=G$, we are done.

Suppose in the following that $G$ is not 3-connected.
Then $G$ has a 2-vertex cut $\{u,v\}$.
By Lemma \ref{property-8}($\romannumeral1$),
$G-\{u,v\}$ has exactly two components $G_1^*$ and $G_2^*$.
Let $G_i'$, $i=1,2$, be the graph obtained from
the $\{u,v\}$-component of $G$ which contains $G_i^*$ by
deleting all possible edges of $G$ which connect $u$ and $v$.
Note that $(G_i')_{sim}$ is the underlying simple graph of $G_i'$.
By Lemma \ref{property-8}($\romannumeral2$),
at least one of $(G_1')_{sim}$ and $(G_2')_{sim}$ is a path,
say $(G_2')_{sim}$.
Let $G'$ be the graph obtained from $G$ by replacing $G_2'$ by $(G_2')_{sim}$.
Since $G$ is cycle-nice, each cycle of $G$ of length two is a nice cycle.
So each member of multiple edges of $G_2'$ remained in $G'$ is admissible.
Since $G$ is claw-free, $G'$ is claw-free.

If $(G_2')_{sim}$ is an odd path,
let $H_1$ be the marked $\{u,v\}$-component of $G$
which contains $G_1^*$. Then $(H_1)_{sim}$ is not $K_2$.
Lemma \ref{property-8}($\romannumeral3$) implies that $H_1$ is cycle-nice.
Note that $G'$ is an even subdivision of $H_1$.
By Lemma \ref{property-4.3}, $H_1$ is claw-free,
and by Lemma \ref{property-4.4}, $(H_1)_{sim}$ is not $W_5$.
Furthermore, by Lemma \ref{property-2}, $H_1$ is 2-connected.
By induction hypothesis,
there exists a sequence $(G_1, G_2, \ldots, G_{k})$ of graphs such that
($\romannumeral1$) $G_1$ is an even cycle, a diamond, $K_4$, or
$\overline C_6$, and $G_{k}=H_1$,
($\romannumeral2$) $G_i$ is an even subdivision or
an odd L-expansion of $G_{i-1}$,
or $G_i$ is obtained from $G_{i-1}$ by replacing some admissible edges of $G_{i-1}$ by some multiple edges, $i=2, 3, \ldots, k$.
If $G_2'$ is simple, we set $G_{k+1}=G$ and $r=k+1$.
If $G_2'$ is not simple, we set $G_{k+1}=G'$, $G_{k+2}=G$ and $r=k+2$.
Then the desired sequence of graphs is obtained.

If $(G_2')_{sim}$ is an even path, write $H_1=G\{V(G_1^*)\}$.
Then Lemma \ref{property-8}($\romannumeral4$) implies that $H_1$ is cycle-nice.
Note that $G'$ is an odd expansion of $H_1$.
By Lemma \ref{property-4.3}, $H_1$ is claw-free,
and by Lemma \ref{property-4.4}, $(H_1)_{sim}$ is not $W_5$.
Note that $H_1$ is 2-connected.

If $uv\in E(G)$, then $G'$ is an odd A-expansion of $H_1$.
Recall that $G$ is cycle-nice and claw-free.
Then $G'$ is cycle-nice and claw-free.
By Lemma \ref{property-4.1},
the underlying simple graph $G_{sim}'$ of $G'$ is a quasi-diamond.
Note that a quasi-diamond can be obtained from a diamond by
an even subdivision at an edge
whose one end vertex is a 2-vertex of the diamond.
When $G$ is simple, if $G_{sim}'$ is a diamond, we set $G_1=G$;
otherwise, we set $G_1$ be a diamond, and $G_2=G$.
When $G$ is not simple, if $G_{sim}'$ is a diamond,
we set $G_1$ be a diamond, $G_2=G$;
Otherwise, we set $G_1$ be a diamond, $G_2=G_{sim}'$, and $G_3=G$.
Then the desired sequence of graphs is obtained.

If there is no edge connecting $u$ and $v$ in $G$,
then $G'$ is an odd L-expansion of $H_1$.
Suppose first that $(H_1)_{sim}$ is $K_2$, then $G_{sim}$ is an even cycle.
If $G$ is simple, we set $G_1=G$;
Otherwise, we set $G_1=G_{sim}$, and $G_2=G$.
Suppose next that $(H_1)_{sim}$ is not $K_2$.
Recall that $H_1$ is 2-connected, claw-free and cycle-nice,
and $(H_1)_{sim}$ is not $W_5$.
By induction hypothesis, there exists a sequence
$(G_1, G_2, \ldots, G_{k})$ of graphs such that
($\romannumeral1$) $G_1$ is an even cycle, a diamond, $K_4$,
or $\overline C_6$, and $G_{k}=H_1$,
($\romannumeral2$) $G_i$ is an even subdivision or an odd L-expansion of $G_{i-1}$, or $G_i$ is obtained from $G_{i-1}$
by replacing some admissible edges of $G_{i-1}$ by some multiple edges, $i=2, 3, \ldots, k$.
Let $G_{k+1}=G'$.
Recall that $(G_2')_{sim}$ is an even path.
If $G_2'$ is simple, then $G_{k+1}=G'=G$.
If $G_2'$ is not simple, we set $G_{k+2}=G$.
Again, the desired sequence of graphs is obtained.
This completes the proof.
\end{proof}

\section*{Acknowledgements}
We would like to thank the anonymous referees for their the constructive comments and kind
suggestions on improving the representation of the paper.

\nocite{*}
\bibliographystyle{abbrvnat}
\bibliography{sample-dmtcs(1)}
\label{sec:biblio}

\end{document}